\documentclass[12pt, reqno]{amsart}
\usepackage{amsmath, amsthm, amscd, amsfonts, amssymb, graphicx, color}
\usepackage[bookmarksnumbered, colorlinks, plainpages]{hyperref}

\hypersetup{colorlinks=true,linkcolor=red, anchorcolor=green, citecolor=cyan, urlcolor=red, filecolor=magenta, pdftoolbar=true}
\newtheorem{theorem}{Theorem}[section]
\newtheorem{lemma}[theorem]{Lemma}

\theoremstyle{definition}

\theoremstyle{remark}
\newtheorem{remark}[theorem]{Remark}
\numberwithin{equation}{section}

\begin{document}

\setcounter{page}{1}


\begin{center}
{\Large \textbf{Modified Stancu-type Dunkl generalization of Sz\'{a}%
sz- Kantorovich-operators}}

\bigskip

\textbf{M. Mursaleen} and \textbf{Md. Nasiruzzaman}

Department of\ Mathematics, Aligarh Muslim University, Aligarh--202002, India%
\\[0pt]

mursaleenm@gmail.com; nasir3489@gmail.com \\[0pt]

\bigskip

\textbf{Abstract}
\end{center}

\parindent=8mm {\footnotesize {In this paper, we introduce a modification of the Sz\'{a}sz-Mirakjan-Kantorovich operators as well as Stancu operators \cite{ckz} (or a Dunkl generalization of modified Sz\'{a}sz-Mirakjan-Kantrovich operators \cite{duman}) which preserve
the linear functions. These types of operators modification enables better error estimation on the interval
$[\frac{1}{2},\infty)$ than the classical Dunkl Sz\'{a}sz-Mirakjan-Kantrovich as well as Stancu operators.
We obtain some approximation results via well known Korovkin's type theorem, weighted Korovkin's type theorem
convergence properties by using the modulus of continuity and the rate of convergence of the operators for
functions belonging to the Lipschitz class.}}\newline












\section{Introduction and preliminaries}

In 1912, S.N Bernstein \cite{sbbl1} introduced the following sequence of
operators $B_{n}:C[0,1]\rightarrow C[0,1]$ defined by
\begin{equation}
B_{n}(f;x)=\sum_{k=0}^{n}\binom{n}{k}x^{k}(1-x)^{n-k}f\left( \frac{k}{n}%
\right) ,~~~~~~~x\in \lbrack 0,1].  \label{s1}
\end{equation}%
for $n\in \mathbb{N}$ and $f\in C[0,1]$.

In 1930 the first Bernstein-Kantorovich operators \cite{kantiv}. And in 1950 Sz\'{a}asz
operators \cite{sbbl4} for $x \geq 0$, defined as follows:
\begin{equation}  \label{s2}
S_n(f;x)=e^{-nx}\sum_{k=0}^\infty \frac{(nx)^k}{k!} f\left(\frac{k}{n}%
\right),~~~~~~~f \in C[0,\infty).
\end{equation}

In recent years, many results about the generalization of Sz\'{a}sz operators have been obtained by several mathematicians
\cite{satk, satik, smah2, smur1, smur3, stk, stk1}.
Recently, Sucu \cite{sbbl9} define a Dunkl analogue of Sz\'{a}sz operators via a generalization of the exponential function given
by \cite{sbbl10} as follows:

\begin{equation}
S_{n}^{\ast}(f;x):= \frac{1}{e_\mu(nx)}\sum_{k=0}^\infty \frac{(nx)^k}{%
\gamma_\mu(k)} f \left(\frac{k+2\mu\theta_k}{n}\right),
\end{equation}
where $x \geq 0,~~~f \in C[0,\infty),\mu \geq 0,~~~n \in \mathbb{N}$\newline
and
\begin{equation}\label{tot1}
e_\mu(x)= \sum_{n=0}^\infty \frac{x^n}{\gamma_\mu(n)}.
\end{equation}
Here
\begin{equation}\label{tot2}
\gamma_\mu(2k)= \frac{2^{2k}k!\Gamma\left(k+\mu+\frac{1}{2}\right)}{%
\Gamma\left(\mu+\frac{1}{2}\right)},~~~\gamma_\mu(2k+1)= \frac{2^{2k+1}k!\Gamma\left(k+\mu+\frac{3}{2}\right)}{%
\Gamma\left(\mu+\frac{1}{2}\right)}.
\end{equation}

There is given a recursion for $\gamma _{\mu }$
\begin{equation*}
\gamma _{\mu }(k+1)=(k+1+2\mu \theta _{k+1})\gamma _{\mu
}(k),~~~~k=0,1,2,\cdots ,
\end{equation*}%
where
\begin{equation*}
\theta _{k}=%
\begin{cases}
0 & \quad \text{if }k\in 2\mathbb{N} \\
1 & \quad \text{if }k\in 2\mathbb{N}+1.%
\end{cases}%
\end{equation*}

For $\mu \geq 0,~~x\geq 0$,  and $f\in C[0,\infty )$,  G\"{u}rhan I\c{c}\"{o}z \cite{ckz} gave a Dunkl generalization of
Kantrovich type integral generalization of Sz\'{a}sz operators as follows:
\begin{equation}\label{sss1}
T_{n}^{\ast}(f;x)=\frac{n}{e_{\mu}(nx)}\sum_{k=0}^{\infty }\frac{%
(nx)^{k}}{\gamma _{\mu }(k)}\int_{\frac{k+2\mu \theta
_{k}}{n}}^{\frac{k+1+2\mu \theta _{k}}{n}}f\left(\frac{nt+\alpha}{n+\beta}\right)\mathrm{d}t ,
\end{equation}%
 where $e_\mu(x)$ and $\gamma_k$ are defined in \cite{sbbl9} by \eqref{tot1},\eqref{tot2}.

\begin{lemma}
\label{vod}

\begin{enumerate}
\item \label{vd1} $T_{n}^{\ast}(1;x)=1$,

\item \label{vd2} $T_{n}^{\ast}(t;x)=\frac{n}{n+\beta}x +\frac{1}{n+\beta}\left(\alpha+\frac{1}{2}\right)$,

\item \label{vd3}$ T_{n}^{\ast}(t^2;x) = \left(\frac{n}{n+\beta}\right)^2
\left(x^2+2 \left(1+\mu\frac{e_\mu(-nx)}{e_\mu(nx)}\right)\frac{x}{n}+\frac{1}{3n^2}\right)
+\frac{2n\alpha}{(n+\beta)^2}\left(x+\frac{1}{2n}\right)+\left(\frac{\alpha}{n+\beta}\right)^2$.
\end{enumerate}
\end{lemma}

Previous studies demonstrate that providing a better error estimation for positive linear operators plays
an important role in approximation theory, which allows us to approximate much faster to the function being
approximated. In \cite{duman, 3, 4}, various better approximation properties of the Sz\'{a}sz-Mirakjan-Kantrovich operators,
Sz\'{a}sz-Mirakjan operators and Sz\'{a}sz-Mirakjan-Beta operators were investigated.
Recently, in \cite{duman, ckz}, by modifying the Dunkl generalization of Sz\'{a}sz-Mirakjan operators,
we have showed that our modified operators have better error estimation than the classical ones.
 In this paper, we apply the Dunkl generalization to the modified Sz\'{a}sz-Mirakjan-Kantorovich operators \cite{duman}
 and a better approximation to Sz\'{a}sz-Mirakjan-Stancu operators \cite{ckz}.

\section{Construction of operators}

We modify the Sz\'{a}sz-Mirakjan-Kantrovich operators \cite{duman} and define a Dunkl generalization of these modified operators
(or a modification of Dunkl generalization of Sz\'{a}sz-Mirakjan-Kantrovich operators\cite{ckz}) as follows: \newline

Let $\{r_n(x)\}$ be a sequence of real-valued continuous functions defined on $[0,\infty)$ with $0 \leq r_n(x) < \infty$ such that
\begin{equation}
r_n(x)=x- \frac{1}{2n},~~~~x \geq \frac{1}{2}~~~\mbox{and}~~~n \in \mathbb{N}.
\end{equation}

Then for any $x \geq \frac{1}{2},~~~\mu \geq 0$ and $n \in \mathbb{N}$ we define

\begin{equation}
K_{n}(f;x)=\frac{n}{e_{\mu }(n r_n(x))}\sum_{k=0}^{\infty }%
\frac{(n r_n(x))^{k}}{\gamma _{\mu}(k)}\int_{\frac{k+2\mu \theta
_{k}}{n}}^{\frac{k+1+2\mu \theta _{k}}{n}}f(t)\mathrm{d}t,
\label{snss1}
\end{equation}%

where $e_{\mu}(x),~~~\gamma_\mu$ are defined in \cite{sbbl9} by \eqref{tot1},\eqref{tot2} and $f \in C_\zeta[0,\infty)  $  with $\zeta \geq 0$ and
\begin{equation}\label{game}
C_\zeta[0,\infty)=\{ f \in C[0,\infty) : \mid f(t) \mid \leq M(1+t)^\zeta,~~~\mbox{for}~~~\mbox{some}~~~M>0, ~~\zeta>0\}.
\end{equation}
 If we take $\mu=0$ in the operator $K_n$ defined by \eqref{snss1}, then the operator $K_n$ reduces to the modified Sz\'{a}sz-Mirakjan-Kantorovich operators given by Oktay Duman et al., \cite{duman}.

\begin{lemma}
\label{snlm1} Let $K_{n}(.~;~.)$ be the operators given by %
\eqref{snss1}. Then for each $x \geq \frac{1}{2}$, we have we have the following identities:

\begin{enumerate}
\item \label{snlm11} $K_{n}(1;x)=1$,

\item \label{snlm12} $K_{n}(t;x)=x$,

\item \label{snlm13} $K_{n}(t^2;x) = x^2+\frac{1}{n}\left(1+2 \mu \frac{e_\mu(-n r_n(x))}{e_\mu(n r_n(x))}\right)x-
\frac{1}{n^2}\left(\frac{5}{12}+ \mu \frac{e_\mu(-n r_n(x))}{e_\mu(n r_n(x))}\right) $.
\end{enumerate}
\end{lemma}

Here we define a  modified Sz\'{a}sz-Mirakjan-Kantrovich Stancu operators \cite{duman} and obtain better approximation results by their Dunkl generalization as follows: \newline
For $x \geq \frac{1}{2},~~~\zeta \geq 0 ,~~~n \in \mathbb{N}$, if $f \in C_\zeta[0,\infty)$ satisfying \eqref{game} with $\zeta \geq 0$, then we define

\begin{equation}
K_{n}^{\ast }(f;x)=\frac{n}{e_{\mu }(n r_n(x))}\sum_{k=0}^{\infty }%
\frac{(n r_n(x))^{k}}{\gamma _{\mu}(k)}\int_{\frac{k+2\mu \theta
_{k}}{n}}^{\frac{k+1+2\mu \theta _{k}}{n}}f\left(\frac{nt+\alpha}{n+\beta}\right)\mathrm{d}t,
\label{snssns1}
\end{equation}%

where $e_{\mu}(x),~~~\gamma_\mu$ are defined in \cite{sbbl9} by \eqref{tot1},\eqref{tot2}. \newline
 If we take $\alpha=\beta=0$ in the operator $K_n^{\ast }$ defined by \eqref{snssns1}, then the operator $K_n^{\ast }$ reduces to operators defined by  \eqref{snss1}. And if we take $\mu=0$, the it reduce the operators defined in \cite{duman}.

\begin{lemma}
\label{snlm1} Let $K_{n}^{\ast }(.~;~.)$ be the operators given by %
\eqref{snssns1}. Then for each $x \geq \frac{1}{2}$, we have we have the following identities:

\begin{enumerate}
\item \label{snlm11} $K_{n}^{\ast }(1;x)=1$,

\item \label{snlm12} $K_{n}^{\ast }(t;x)=\frac{n}{n+\beta}x+\frac{\alpha}{n+\beta}$,

\item \label{snlm13} $K_{n}^{\ast }(t^2;x) = \left(\frac{n}{n+\beta}\right)^2x^2+
\frac{n}{(n+\beta)^2} \left(1+2\alpha+2 \mu \frac{e_\mu(-n r_n(x))}{e_\mu(n r_n(x))}\right)x\newline
+\frac{1}{(n+\beta)^2} \left(\alpha^2-\left(\frac{5}{12}+\mu \frac{e_\mu(-n r_n(x))}{e_\mu(n r_n(x))}\right)\right) $.
\end{enumerate}
\end{lemma}


\begin{lemma}\label{mtn}
\label{snlm2} Let the operators $K_{n}^{\ast }(.~;~.)$ be given by %
\eqref{snssns1}. Then for each $x \geq \frac{1}{2}$, we have

\begin{enumerate}
\item \label{snlm22} $K_{n}^{\ast}(t-x;x)=\left(\frac{n}{n+\beta}-1\right)x+\frac{\alpha}{n+\beta}$,

\item \label{snlm23} $K_{n}^{\ast}((t-x)^2;x) = \frac{1}{(n+\beta)^2}\left\{\beta^2x^2+\left(n-2 \alpha \beta+2n\mu\frac{e_\mu(-nr_n(x))}{e_\mu(nr_n(x))}\right)x+\alpha^2-\left(\frac{5}{12}+\mu\frac{e_\mu(-nr_n(x))}{e_\mu(nr_n(x))}\right)\right\}$.
\end{enumerate}
\end{lemma}



\section{Main results}

We obtain the Korovkin's type approximation properties for our operators
defined by \eqref{snss1}.\newline

Let $C_{B}(\mathbb{R^{+}})$ be the set of all bounded and continuous
functions on $\mathbb{R^{+}}=[0,\infty )$, which is linear normed space with
\begin{equation*}
\parallel f\parallel _{C_{B}}=\sup_{x\geq 0}\mid f(x)\mid .
\end{equation*}%
Let
\begin{equation*}
H:=\{f:x\in \lbrack 0,\infty ),\frac{f(x)}{1+x^{2}}~~~\mbox{is}~~~%
\mbox{convergent}~~~\mbox{as}~~~x\rightarrow \infty \}.
\end{equation*}


\begin{theorem}
\label{snth1} Let $K_{n}^{\ast}(t;x)$  be the operators defined by \eqref{snssns1}.
Then for any function $f\in C_\zeta[0,\infty )\cap H,~~~\zeta \geq 2$,
\begin{equation*}
\lim_{n\rightarrow \infty }K_{n}^{\ast }(f;x)=f(x)
\end{equation*}%
is uniformly on each compact subset of $[0,\infty )$, where $x \in \big{[}\frac{1}{2},b),~~~b>\frac{1}{2}$.
\end{theorem}

\begin{proof}
The proof is based on Lemma \ref{snssns1} and well known Korovkin's theorem regarding the
convergence of a sequence of linear and positive operators, so it is enough
to prove the conditions
\begin{equation*}
\lim_{n\rightarrow \infty }{K}_{n}^{\ast
}((t^{j};x)=x^{j},~~~j=0,1,2,~~~\{\mbox{as}~n\rightarrow \infty \}
\end{equation*}%
uniformly on $[0,1]$.\newline
Clearly $\frac{1}{n}\rightarrow
0~~(n\rightarrow \infty )$ we have

\begin{equation*}
\lim_{n \to \infty}{K}_{n}^{\ast}(t;x)=x,~~~\lim_{n \to \infty}{K}%
_{n}^{\ast}(t^2;x)=x^2.
\end{equation*}
Which complete the proof.
\end{proof}


We recall the weighted spaces of the functions on $\mathbb{R}^{+}$, which
are defined as follows:
\begin{eqnarray*}
P_{\rho }(\mathbb{R}^{+}) &=&\left\{ f:\mid f(x)\mid \leq M_{f}\rho
(x)\right\} , \\
Q_{\rho }(\mathbb{R}^{+}) &=&\left\{ f:f\in P_{\rho }(\mathbb{R}^{+})\cap
C[0,\infty )\right\} , \\
Q_{\rho }^{k}(\mathbb{R}^{+}) &=&\left\{ f:f\in Q_{\rho }(\mathbb{R}^{+})~~~%
\mbox{and}~~~\lim_{x\rightarrow \infty }\frac{f(x)}{\rho (x)}=k(k~~~\mbox{is}%
~~~\mbox{a}~~~\mbox{constant})\right\} ,
\end{eqnarray*}%
where $\rho (x)=1+x^{2}$ is a weight function and $M_{f}$ is a constant
depending only on $f$. Note that $Q_{\rho }(\mathbb{R}^{+})$ is a normed
space with the norm $\parallel f\parallel _{\rho }=\sup_{x\geq 0}\frac{\mid
f(x)\mid }{\rho (x)}$.

\begin{theorem}
\label{soo2} Let $K_{n}^{\ast}(t;x)$  be the operators defined by \eqref{snssns1} acting from
$Q_{\rho }(\mathbb{R}^{+})\to P_{\rho }(\mathbb{R}^{+}) $ and satisfying the condition
\begin{equation*}
\lim_{n\to \infty} \parallel K_{n}^{\ast}(\rho^\tau)-\rho^\tau \parallel_\varphi=0,~~~~\tau=0,1,2.
\end{equation*}
 Then for any function $f\in Q_{\rho }^{k}(\mathbb{R}^{+})$, we have
\begin{equation*}
\lim_{n\to \infty} \parallel K_{n}^{\ast}(f;x)-f \parallel_\varphi=0.
\end{equation*}
\end{theorem}
\begin{proof}
From \cite{nbn}, we can easily lead to desired result.
\end{proof}

\begin{theorem}
\label{snth2} Let $K_{n}^{\ast}(t;x)$  be the operators defined by \eqref{snssns1}. Then for each
function $f \in Q^k_\rho(\mathbb{R}^+)$ we have
\begin{equation*}
\lim_{n\to \infty} \parallel K_{n}^{\ast}(f;x)-f \parallel_\rho=0.
\end{equation*}
\end{theorem}

\begin{proof}
From Lemma \ref{snlm1} and Theorem \ref{soo2} for $\tau =0$, the first condition is fulfilled. Therefore
\begin{equation*}
\lim_{n\to \infty} \parallel K_{n}^{\ast}(1;x)-1 \parallel_\rho=0.
\end{equation*}
Similarly From Lemma \ref{snlm1} and Theorem \ref{soo2} for $\tau =1, 2$ we have that
\begin{eqnarray*}
\sup_{x \in[0,\infty)}\frac{\mid K_{n}^{\ast}(t;x)-x\mid }{1+x^2} &\leq & \big{|}\frac{n}{n+\beta}-1\big{|}\sup_{x \in[0,\infty)}\frac{x}{1+x^2}
+\frac{\alpha}{n+\beta}\sup_{x \in[0,\infty)}\frac{1}{1+x^2}\\
&=&\frac{\beta}{2(n+\beta)}+\frac{\alpha}{n+\beta},
\end{eqnarray*}%
which imply that
\begin{equation*}
\lim_{n\to \infty} \parallel K_{n}^{\ast}(t;x)-x \parallel_\rho=0.
\end{equation*}

\begin{eqnarray*}
\sup_{x \in[0,\infty)}\frac{\mid K_{n}^{\ast}(t^2;x)-x^2\mid}{1+x^2} &\leq & \frac{2n\beta+\beta^2}{(n+\beta)^2}\sup_{x \in[0,\infty)}\frac{x^2}{1+x^2}\\
&+&\frac{n}{(n+\beta)^2}\left(1+2 \alpha+2\mu\frac{e_\mu(-nr_n(x))}{e_\mu(nr_n(x))}\right)\sup_{x \in[0,\infty)}\frac{x}{1+x^2}\\
&+&\frac{1}{(n+\beta)^2}\left( \alpha^2-\left(\frac{5}{12}+\mu\frac{e_\mu(-nr_n(x))}{e_\mu(nr_n(x))}\right)\right)\sup_{x \in[0,\infty)}\frac{1}{1+x^2}\\
&=&\frac{1}{(n+\beta)^2}\left\{ \beta^2+2n\beta+\left(\frac{1}{2}+\alpha+\mu\frac{e_\mu(-nr_n(x))}{e_\mu(nr_n(x))}\right)n\right\}\\
&+&\frac{1}{(n+\beta)^2}\left\{\alpha^2-\left(\frac{5}{12}+\mu\frac{e_\mu(-nr_n(x))}{e_\mu(nr_n(x))}\right)\right\},
\end{eqnarray*}%
which imply that
\begin{equation*}
\lim_{n\to \infty} \parallel K_{n}^{\ast}(t^2;x)-x^2 \parallel_\rho=0.
\end{equation*}
This complete the proof.
\end{proof}

\section{\textbf{Rate of Convergence}}

Here we calculate the rate of convergence of operators \eqref{snss1} by
means of modulus of continuity and Lipschitz type maximal functions.

Let $f\in C_B[0,\infty ]$, the space of all bounded and continuous functions on $[0,\infty)$ and $x \geq \frac{1}{2}$. Then for $\delta>0$, the modulus of continuity of $f$ denoted by $\omega
(f,\delta )$ gives the maximum oscillation of $f$ in any interval of length
not exceeding $\delta >0$ and it is given by
\begin{equation}
\omega (f,\delta )=\sup_{\mid t-x\mid \leq \delta }\mid f(t)-f(x)\mid
,~~~t\in \lbrack 0,\infty ).  \label{snson1}
\end{equation}%
It is known that $\lim_{\delta\rightarrow 0+}\omega (f,\delta )=0$ for $%
f\in C_B[0,\infty )$ and for any $\delta >0$ one has
\begin{equation}
\mid f(t)-f(x)\mid \leq \left( \frac{\mid t-x\mid }{\delta }+1\right) \omega
(f,\delta ).  \label{snson2}
\end{equation}

\begin{theorem}
Let $K_{n}(.~;~.)$ be the operators defined by \eqref{snss1}. Then for $%
f\in C_B[0,\infty),~~x\geq \frac{1}{2}$ and $n \in \mathbb{N}$ we have
\begin{equation*}
\mid K_{n}(f;x)-f(x)\mid \leq 2\omega\left(f;\delta_{n,x}\right),
\end{equation*}%
where from Lemma \ref{snlm1} we have
$$\delta_{n,x}=\sqrt{K_{n}((t-x)^2;x)}=   \sqrt{\left(1+2\mu \frac{e_\mu(-nr_n(x))}{e_\mu(nr_n(x))}\right)\frac{x}{n}
-\frac{1}{n^2}\left(\frac{5}{12}+\frac{e_\mu(-nr_n(x))}{e_\mu(nr_n(x))}\right)}.$$
\end{theorem}

\begin{proof}
We prove it by using \eqref{snson1}, \eqref{snson2} and Cauchy-Schwarz
inequality.\newline
$\mid K_{n}(f;x)-f(x)\mid $
\begin{eqnarray*}
&\leq &\frac{n}{e_{\mu}(n r_n(x))}\sum_{k=0}^{\infty }\frac{%
(n r_n(x))^{k}}{\gamma _{\mu}(k)}\int_{\frac{k+2\mu \theta _{k}}{%
n}}^{\frac{k+1+2\mu \theta _{k}}{n}}\mid f(t)-f(x)\mid
\mathrm{d}t \\
&\leq &\left\{\frac{n}{e_{\mu}(n r_n(x))}\sum_{k=0}^{\infty }\frac{%
(n r_n(x))^{k}}{\gamma _{\mu}(k)}\int_{\frac{k+2\mu \theta _{k}}{%
n}}^{\frac{k+1+2\mu \theta _{k}}{n}}\left( 1+\frac{1}{%
\delta }\mid t-x\mid \right) \mathrm{d}t\right\}\omega (f;\delta ) \\
&=&\left\{ 1+\frac{1}{\delta }\left( \frac{n}{e_{\mu}(n r_n(x))}\sum_{k=0}^{\infty }\frac{%
(n r_n(x))^{k}}{\gamma _{\mu}(k)}\int_{\frac{k+2\mu \theta _{k}}{%
n}}^{\frac{k+1+2\mu \theta _{k}}{n}}\mid t-x\mid \mathrm{d}t\right) \right\} \omega (f;\delta ) \\
&\leq &\left\{ 1+\frac{1}{\delta }\left( \frac{n}{e_{\mu}(n r_n(x))}\sum_{k=0}^{\infty }\frac{%
(n r_n(x))^{k}}{\gamma _{\mu}(k)}\int_{\frac{k+2\mu \theta _{k}}{%
n}}^{\frac{k+1+2\mu \theta _{k}}{n}}(t-x)^{2}\mathrm{d}t\right) ^{\frac{1}{2}}\left( K_{n}^{\ast
}(1;x)\right) ^{\frac{1}{2}}\right\} \omega (f;\delta ) \\
&=&\left\{ 1+\frac{1}{\delta }\left( K_{n}(t-x)^{2};x\right) ^{%
\frac{1}{2}}\right\} \omega (f;\delta ) \\
&&
\end{eqnarray*}%
if we choose $\delta=\delta _{n,x}$, the proof is complete.
\end{proof}
\begin{theorem}\label{numb}
Let $K_{n}^{\ast }(.~;~.)$ be the operators defined by \eqref{snssns1}. Then for $%
f\in C_B[0,\infty),~~x\geq \frac{1}{2}$ and $n \in \mathbb{N}$ we have
\begin{equation*}
\mid K_{n}^{\ast }(f;x)-f(x)\mid \leq 2\omega\left(f;\delta_{n,x}^{\ast}\right),
\end{equation*}%
where
$C_B[0,\infty)$ is the space of uniformly continuous bounded functions
on $\mathbb{R}^+,~~~$  $\omega(f,\delta)$ is the modulus of continuity of the
function $f \in C_B[0,\infty)$ defined in \eqref{snson1} and
\begin{equation}\label{pop}
\delta_{n,x}^{\ast}=\sqrt{\frac{1}{(n+\beta)^2}\left\{\beta^2x^2+\left(n-2 \alpha \beta+2n\mu\frac{e_\mu(-nr_n(x))}{e_\mu(nr_n(x))}\right)x+\alpha^2-\left(\frac{5}{12}+\mu\frac{e_\mu(-nr_n(x))}{e_\mu(nr_n(x))}\right)\right\}}.
\end{equation}
\end{theorem}

\begin{proof}
We prove it by using \eqref{snson1}, \eqref{snson2} and Cauchy-Schwarz
inequality we can easily get
$ $
\begin{eqnarray*}
\mid K_{n}^{\ast }(f;x)-f(x)\mid
&\leq &\left\{ 1+\frac{1}{\delta }\left( K_{n}^{\ast }(t-x)^{2};x\right) ^{%
\frac{1}{2}}\right\} \omega (f;\delta^{\ast} )
\end{eqnarray*}%
if we choose $\delta^{\ast}=\delta _{n,x}^{\ast}$ and by applying the result (\ref{snlm23}) of Lemma \ref{mtn} complete the proof.
\end{proof}

\begin{remark}
For the operators $T_{n}^{\ast}(.~;~.)$ defined by \eqref{sss1} we may write that, for every $%
f\in C_B[0,\infty),~~x\geq 0$ and $n \in \mathbb{N}$
\begin{equation}\label{nb}
\mid T_{n}^{\ast}(f;x)-f(x)\mid \leq 2\omega\left(f;\lambda_{n,x}\right),
\end{equation}%
where from Lemma \ref{vod} we have

\begin{equation}\label{vb}
\lambda_{n,x}=\sqrt{T_{n}^{\ast}((t-x)^2;x)}=
\frac{1}{(n+\beta)^2}\left\{\beta^2x^2+\left(n-\beta(2\alpha+1)+2n\mu\frac{e_\mu(-nx)}{e_\mu(nx)}\right)x+\alpha^2+\alpha+\frac{1}{3}\right\}.
\end{equation}
\end{remark}
Now we claim that the error estimation in Theorem \ref{numb} is better than that of \eqref{nb} provided $f \in C_B[0,\infty)$ and $x\geq \frac{1}{2}$.
Indeed, for $x \geq \frac{1}{2},~~\mu\geq 0$ and $n \in \mathbb{N}$, it is guarantees that
\begin{equation}
K_{n}^{\ast}((t-x)^2;x)\leq T_{n}^{\ast}((t-x)^2;x),
\end{equation}
where $K_{n}^{\ast}((t-x)^2;x)$ and $T_{n}^{\ast}((t-x)^2;x)$ are defined in Lemma \ref{mtn} and in \eqref{vb}.
If we put $\alpha=\beta=0$ then clearly
\begin{equation}\label{cvc}
\left(1+2\mu\frac{e_\mu(-nr_n(x))}{e_\mu(nr_n(x))}\right)\frac{x}{n}-\frac{1}{n^2}\left(\frac{5}{12}+\mu\frac{e_\mu(-nr_n(x))}{e_\mu(nr_n(x))}\right)
\leq \left(1+2\mu\frac{e_\mu(-nx)}{e_\mu(nx)}\right)\frac{x}{n}+\frac{1}{3n^2}.
\end{equation}
Again if we put $\mu=0$, then the result in \cite{duman} by equation (3.6) is obtained as
\begin{equation}\label{cvc}
\frac{x}{n}-\frac{5}{12n^2} \leq \frac{x}{n}+\frac{1}{3n^2}.
\end{equation}


Now we give the rate of convergence of the operators ${K}_{n}^*(f;x) $
defined in \eqref{snssns1} in terms of the elements of the usual Lipschitz
class $Lip_{M}(\nu )$.

Let $f\in C_B[0,\infty )$, $M>0$ and $0<\nu \leq 1$. The class $Lip_{M}(\nu )$
is defined as
\begin{equation}
Lip_{M}(\nu )=\left\{ f:\mid f(\zeta _{1})-f(\zeta _{2})\mid \leq M\mid
\zeta _{1}-\zeta _{2}\mid ^{\nu }~~~(\zeta _{1},\zeta _{2}\in \lbrack
0,\infty ))\right\}  \label{snn1}
\end{equation}

\begin{theorem}
\label{snsn1} Let $K_{n}^{\ast}(.~;~.)$ be the operator defined in \eqref{snssns1}.%
Then for each $f\in Lip_{M}(\nu ),~~(M>0,~~~0<\nu \leq 1)$ satisfying %
\eqref{snn1} we have
\begin{equation*}
\mid K_{n}^{\ast}(f;x)-f(x)\mid \leq M \left(\delta_{n,x}^{\ast}\right)^{\frac{\nu}{2%
}}
\end{equation*}
where $\delta_{n,x}^{\ast}$ is given in Theorem \ref{pop}.
\end{theorem}

\begin{proof}
We prove it by using \eqref{snn1} and H\"{o}lder inequality.
\begin{eqnarray*}
\mid K_{n}^{\ast }(f;x)-f(x)\mid &\leq &\mid K_{n}^{\ast
}(f(t)-f(x);x)\mid \\
&\leq &K_{n}^{\ast }\left( \mid f(t)-f(x)\mid ;x\right) \\
&\leq &\mid MK_{n}^{\ast }\left( \mid t-x\mid ^{\nu };x\right) .
\end{eqnarray*}%
Therefore\newline

$\mid K_{n}^{\ast}(f;x)-f(x) \mid$
\begin{eqnarray*}
&\leq & M \frac{n}{e_{\mu}(n r_n(x))}\sum_{k=0}^{\infty }\frac{%
(n r_n(x))^{k}}{\gamma _{\mu}(k)}\int_{\frac{k+2\mu \theta _{k}}{%
n}}^{\frac{k+1+2\mu \theta _{k}}{n}}\mid t-x \mid^\nu \mathrm{d}t \\
& \leq & M \frac{n}{e_{\mu}(nr_n(x))}\sum_{k=0}^\infty \left(\frac{%
(n r_n(x))^{k}}{\gamma_{\mu}(k)}\right)^{\frac{2-\nu}{2}} \\
& \times & \left(\frac{(n r_n(x))^{k}}{\gamma_{\mu}(k)}\right)^{\frac{\nu}{2}%
} \int_{\frac{k+2\mu \theta_k}{n}}^{\frac{k+1+2\mu \theta_k}{%
n}}\mid t-x \mid^\nu \mathrm{d}t \\
& \leq & M \left(\frac{n}{e_{\mu}(n r_n(x))}\sum_{k=0}^{\infty }\frac{%
(n r_n(x))^{k}}{\gamma _{\mu}(k)}\int_{\frac{k+2\mu \theta _{k}}{%
n}}^{\frac{k+1+2\mu \theta _{k}}{n}} \mathrm{d}t\right)^{%
\frac{2-\nu}{2}} \\
& \times & \left(\frac{n}{e_{\mu}(n r_n(x))}\sum_{k=0}^{\infty }\frac{%
(n r_n(x))^{k}}{\gamma _{\mu}(k)}\int_{\frac{k+2\mu \theta _{k}}{%
n}}^{\frac{k+1+2\mu \theta _{k}}{n}}\mid t-x \mid^2
\mathrm{d}t \right)^{\frac{\nu}{2}} \\
& = & M \left(K_{n}^{\ast}(t-x)^2;x\right)^{\frac{\nu}{2}}.
\end{eqnarray*}
Which complete the proof.
\end{proof}

Let $C_{B}[0,\infty )$ denote the space of all bounded and continuous
functions on $\mathbb{R}^{+}=[0,\infty )$ and
\begin{equation}
C_{B}^{2}(\mathbb{R}^{+})=\{g\in C_{B}(\mathbb{R}^{+}):g^{\prime },g^{\prime
\prime }\in C_{B}(\mathbb{R}^{+})\},  \label{snt2}
\end{equation}%
with the norm
\begin{equation}
\parallel g\parallel _{C_{B}^{2}(\mathbb{R}^{+})}=\parallel g\parallel
_{C_{B}(\mathbb{R}^{+})}+\parallel g^{\prime }\parallel _{C_{B}(\mathbb{R}%
^{+})}+\parallel g^{\prime \prime }\parallel _{C_{B}(\mathbb{R}^{+})},
\label{snt1}
\end{equation}%
also
\begin{equation}
\parallel g\parallel _{C_{B}(\mathbb{R}^{+})}=\sup_{x\in \mathbb{R}^{+}}\mid
g(x)\mid .  \label{snt3}
\end{equation}

\begin{theorem}
\label{snsn2} Let $K_{n}^{\ast}(.~;~.)$ be the operator defined in \eqref{snssns1}%
. Then for any $g \in C_B^2(\mathbb{R}^+)$ we have
\begin{equation*}
\mid K_{n}^{\ast}(f;x)-f(x)\mid \leq \left\{ \left( \left(\frac{n}{n+\beta}-1\right)x+\frac{\alpha}{n+\beta}\right)
+\frac{\delta_{n,x}^{\ast}}{2}\right\} \parallel g\parallel_{C_B^2(\mathbb{R}^+)},
\end{equation*}
where $\delta_{n,x}^{\ast}$ is given in Theorem \ref{pop}.
\end{theorem}

\begin{proof}
Let $g\in C_{B}^{2}(\mathbb{R}^{+})$, then by using the generalized mean
value theorem in the Taylor series expansion we have
\begin{equation*}
g(t)=g(x)+g^{\prime }(x)(t-x)+g^{\prime \prime }(\psi )\frac{(t-x)^{2}}{2}%
,~~~\psi \in (x,t).
\end{equation*}%
By applying linearity property on $K_{n}^{\ast },$ we have
\begin{equation*}
K_{n}^{\ast }(g,x)-g(x)=g^{\prime }(x)K_{n}^{\ast }\left( (t-x);x\right)
+\frac{g^{\prime \prime }(\psi )}{2}K_{n}^{\ast }\left( (t-x)^{2};x\right),
\end{equation*}%
which imply that\newline
$\mid K_{n}^{\ast }(g;x)-g(x)\mid \leq \left( \left(\frac{n}{n+\beta}-1\right)x+\frac{\alpha}{n+\beta}\right) \parallel g^{\prime }\parallel _{C_{B}(\mathbb{R}^{+})}$\newline

$+\left\{ \frac{1}{(n+\beta)^2}\left(\beta^2x^2+\left(n-2 \alpha \beta+2n\mu\frac{e_\mu(-nr_n(x))}{e_\mu(nr_n(x))}\right)x+\alpha^2-\left(\frac{5}{12}+\mu\frac{e_\mu(-nr_n(x))}{e_\mu(nr_n(x))}\right)\right)\right\} \frac{\parallel g^{\prime \prime
}\parallel _{C_{B}(\mathbb{R}^{+})}}{2}$\newline.

From \eqref{snt1} we have ~~~$\parallel g^{\prime }\parallel
_{C_{B}[0,\infty )}\leq \parallel g\parallel _{C_{B}^{2}[0,\infty )}$.%
\newline
$\mid K_{n}^{\ast }(g;x)-g(x)\mid \leq \left( \left(\frac{n}{n+\beta}-1\right)x+\frac{\alpha}{n+\beta}\right)\parallel g\parallel _{C_{B}^{2}(\mathbb{R}^{+})}$\newline

$+\left\{ \frac{1}{(n+\beta)^2}\left(\beta^2x^2+\left(n-2 \alpha \beta+2n\mu\frac{e_\mu(-nr_n(x))}{e_\mu(nr_n(x))}\right)x+\alpha^2-\left(\frac{5}{12}+\mu\frac{e_\mu(-nr_n(x))}{e_\mu(nr_n(x))}\right)\right)\right\} \frac{\parallel g\parallel
_{C_{B}^{2}(\mathbb{R}^{+})}}{2}$.\newline
This completes the proof from \ref{snlm23} of Lemma \ref{snlm2}.
\end{proof}

The Peetre's $K$-functional is defined by
\begin{equation}
K_{2}(f,\delta )=\inf_{C_{B}^{2}(\mathbb{R}^{+})}\left\{ \left( \parallel
f-g\parallel _{C_{B}(\mathbb{R}^{+})}+\delta \parallel g^{\prime \prime
}\parallel _{C_{B}^{2}(\mathbb{R}^{+})}\right) :g\in \mathcal{W}^{2}\right\}
,  \label{snzr1}
\end{equation}%
where
\begin{equation}
\mathcal{W}^{2}=\left\{ g\in C_{B}(\mathbb{R}^{+}):g^{\prime },g^{\prime
\prime }\in C_{B}(\mathbb{R}^{+})\right\} .  \label{snzr2}
\end{equation}%
There exits a positive constant $C>0$ such that $K_{2}(f,\delta )\leq
C\omega _{2}(f,\delta ^{\frac{1}{2}}),~~\delta >0$, where the second order
modulus of continuity is given by
\begin{equation}
\omega _{2}(f,\delta ^{\frac{1}{2}})=\sup_{0<h<\delta ^{\frac{1}{2}%
}}\sup_{x\in \mathbb{R}^{+}}\mid f(x+2h)-2f(x+h)+f(x)\mid .  \label{snzr3}
\end{equation}

\begin{theorem}
\label{snsn3} Let $K_{n}^{\ast }(.~;~.)$ be the operator defined in %
\eqref{snss1} and $C_{B}[0,\infty )$ be the space of all bounded and
continuous functions on $\mathbb{R}^{+}$. Then for $x\geq \frac{1}{2}$  and $%
f\in C_{B}(\mathbb{R}^{+})$ we have\newline
$\mid K_{n}^{\ast }(f;x)-f(x)\mid $\newline

$\leq 2M\left\{ \omega _{2}\left( f;\sqrt{\frac{\left(\frac{2n}{n+\beta}-2\right)x+\frac{2\alpha}{n+\beta}+
\delta_{n,x}^{\ast}}{4}}\right) +\min
\left( 1,\frac{\left(\frac{2n}{n+\beta}-2\right)x+\frac{2\alpha}{n+\beta}+
\delta_{n,x}^{\ast}}{4}\right) \parallel f\parallel _{C_{B}(\mathbb{R}%
^{+})}\right\} $,\newline
where $M$ is a positive constant, $\delta_{n,x}^{\ast}$ is given in Theorem \ref%
{pop} and $\omega _{2}(f;\delta )$ is the second order modulus of
continuity of the function $f$ defined in \eqref{snzr3}.
\end{theorem}

\begin{proof}
We prove this by using the Theorem \eqref{snsn2}
\begin{eqnarray*}
\mid K_{n}^{\ast}(f;x)-f(x)\mid &\leq & \mid K_{n}^{\ast}(f-g;x)\mid+\mid
K_{n}^{\ast}(g;x)-g(x)\mid+\mid f(x)-g(x)\mid \\
&\leq & 2 \parallel f-g \parallel_{C_B(\mathbb{R}^+)}+ \frac{\delta_{n,x}^{\ast}}{2}%
\parallel g \parallel_{C_B^2(\mathbb{R}^+)} \\
&+ & \left( \left(\frac{n}{n+\beta}-1\right)x+\frac{\alpha}{n+\beta}\right)%
\parallel g \parallel_{C_B(\mathbb{R}^+)}
\end{eqnarray*}
From \eqref{snt1} clearly we have ~~~$\parallel g
\parallel_{C_B[0,\infty)}\leq \parallel g \parallel_{C_B^2[0,\infty)}$.%
\newline
Therefore,
\begin{equation*}
\mid K_{n}^*(f;x)-f(x)\mid \leq 2 \left(\parallel f-g \parallel_{C_B(%
\mathbb{R}^+)}+\frac{\left(\frac{2n}{n+\beta}-2\right)x+\frac{2\alpha}{n+\beta}+
\delta_{n,x}^{\ast}}{4}\parallel g \parallel_{C_B^2(\mathbb{R}^+)}\right),
\end{equation*}
where $\delta_{n,x}^{\ast}$ is given in Theorem \ref{pop}.\newline

By taking infimum over all $g\in C_{B}^{2}(\mathbb{R}^{+})$ and by using %
\eqref{snzr1}, we get
\begin{equation*}
\mid K_{n}^{\ast }(f;x)-f(x)\mid \leq 2K_{2}\left( f;\frac{\left(\frac{2n}{n+\beta}-2\right)x+\frac{2\alpha}{n+\beta}+
\delta_{n,x}^{\ast}}{4}.%
\right)
\end{equation*}%
Now for an absolute constant $D>0$ in \cite{scc1} we use the relation
\begin{equation*}
K_{2}(f;\delta )\leq D\{\omega _{2}(f;\sqrt{\delta })+\min (1,\delta
)\parallel f\parallel \}.
\end{equation*}%
This complete the proof.
\end{proof}

\section{Concluding remarks and observations}
Motivated essentially by the recent investigation of dunkl type generalization of Stancu operators by G\"{u}rhan I\c{c}\"{o}z \cite{ckz}
and modified Sz\'{a}sz-Mirakjan-Kantrovich operators by Oktay Duman \cite{duman}, we have proved several approximation results.  The operators defined in \eqref{snssns1}, modifying the recent investigation of G\"{u}rhan I\c{c}\"{o}z \cite{ckz} and also if we take
$\alpha=\beta=0$ in operators defined in \eqref{snssns1}, then it reduces to the operators \eqref{snss1}, which becomes a dunkl generalization of the paper investigated by Oktay Duman \cite{duman}. We have successfully extended these results and modifying the results of papers (\cite{ckz, duman}).
Several other related results have also been considered.\newline

\textbf{Acknowledgement.} Second author (MN) acknowledges the financial sup-
port of University Grants Commission (Govt. of Ind.) for awarding BSR (Basic
Scientific Research) Fellowship.

\end{document}